\documentclass{amsart}
\usepackage{amscd,amsmath,amssymb,amsthm,amsfonts,epsfig,graphics}

\theoremstyle{theorem}
\newtheorem{theorem}{Theorem}
\theoremstyle{definition}
\newtheorem{definition}{Definition}
\newtheorem{prop}{Proposition}
\newtheorem{remark}{Remark}

\title{Slant and semi-slant submanifolds in metallic Riemannian manifolds}
\author{Cristina E. Hretcanu and Adara M. Blaga}
     \keywords{Metallic Riemannian structure, Golden Riemannian structure, almost product structure, induced structure on submanifold, slant submanifold, semi-slant submanifold.}
     \subjclass[2010]{53B20, 53B25, 53C15}

     \numberwithin{equation}{section}

\begin{document}

     \begin{abstract}
     The aim of our paper is to focus on some properties of slant and semi-slant submanifolds of metallic Riemannian manifolds. We give some characterizations for submanifolds to be slant or semi-slant submanifolds in metallic or Golden Riemannian manifolds and we obtain integrability conditions for the distributions involved in the semi-slant submanifolds of Riemannian manifolds endowed with metallic or Golden Riemannian structures. Examples of semi-slant submanifolds of the metallic and Golden Riemannian manifolds are given.
     \end{abstract}
     \maketitle

   \section{Introduction} \label{intro}

Since B.Y. Chen defined slant submanifolds in complex manifolds (\cite{Chen3},\cite{Chen4}) in the early 1990's, the differential geometry of slant submanifolds has shown an increasing development. Then, many authors have studied slant submanifolds in different kind of manifolds, such as: slant submanifolds in almost contact metric manifolds (A. Lotta (\cite{Lotta})), in Sasakian manifolds (J.L. Cabrerizo \textit{et al.} (\cite{Cabrerizo1},\cite{Cabrerizo2})), in para-Hermitian manifold (P. Alegre, A. Carriazo (\cite{Alegre})) and in almost product Riemannian manifolds (B. Sahin (\cite{Sahin}), M. At\c{c}eken (\cite{Atceken1},\cite{Atceken2})).

The notion of slant submanifold was generalized by semi-slant submanifold, pseudo-slant submanifold and bi-slant submanifold, respectively, in different types of differentiable manifolds. The semi-slant submanifold of almost Hermitian manifold was introduced by N. Papagiuc (\cite{Papaghiuc}). A. Cariazzo \textit{et al.} (\cite{Carriazo}) defined and studied bi-slant immersion in almost Hermitian manifolds and pseudo-slant submanifold in almost Hermitian manifolds. The pseudo-slant submanifold in Kenmotsu or nearly Kenmotsu manifolds (\cite{Atceken3},\cite{Atceken4}), in LCS-manifolds (\cite{Atceken5}) or in locally decomposable Riemannian manifolds (\cite{Atceken6}) were studied by M. At\c{c}eken \textit{et al.}  Moreover, many examples of semi-slant, pseudo-slant and bi-slant submanifolds were built by most of the authors.

Semi-slant submanifolds are particular cases of bi-slant submanifolds, defined and studied by A. Cariazzo (\cite{Carriazo}). The geometry of slant and semi-slant submanifolds in metallic Riemannian manifolds is related by the properties of slant and semi-slant submanifolds in almost product Riemannian manifolds, studied in (\cite{Atceken1},\cite{Li&Liu},\cite{Sahin}).

The notion of Golden structure on a Riemannian manifold was introduced for the first time by C.E. Hretcanu and M. Crasmareanu in (\cite{CrHr}). Moreover, the authors investigated the properties of a Golden structure related to the almost product structure and of submanifolds in Golden Riemannian manifolds (\cite{Hr2},\cite{Hr3}).  Examples of Golden and product-shaped hypersurfaces in real space forms were given in (\cite{CrHrMu}). The Golden structure was generalized as metallic structures, defined on Riemannian manifolds in (\cite{Hr4}). A.M. Blaga studied the properties of the conjugate connections by a Golden structure and expressed their virtual and structural tensor fields and their behavior on invariant distributions. Also, she studied the impact of the duality between the Golden and almost product structures on Golden and product conjugate connections (\cite{Blaga}). The properties of the metallic conjugate connections were studied by A.M. Blaga and C.E. Hretcanu in (\cite{Blaga2}) where the virtual and structural tensor fields were expressed and their behavior on invariant distributions was analyzed.

Recently, the connection adapted on the almost Golden Riemannian structure was studied by F. Etayo \textit{et al.} in (\cite{Etayo}). Some properties regarding the integrability of the Golden Riemannian structures was investigated by A. Gezer \textit{et al.} in (\cite{Gezer}).

The metallic structure $J$ is a polynomial structure (which was generally defined by S.I. Goldberg \textit{et al.} in (\cite{Goldberg1},\cite{Goldberg2}), inspired by the metallic number $\sigma _{p,q}=\frac{p+\sqrt{p^{2}+4q}}{2}$, which is the positive solution of the equation $x^{2}-px-q=0$, for positive integer values of $p$ and $q$. These $\sigma _{p,q}$ numbers are members of the {\it metallic means family} or {\it metallic proportions} (as generalizations of the Golden number $\phi=\frac{1+\sqrt{5}}{2}=1.618...$), introduced by Vera W. de Spinadel (\cite{Spinadel}). Some examples of the members of the metallic means family are {\it the Silver mean, the Bronze mean, the Copper mean, the Nickel mean} and many others.

The purpose of the present paper is to investigate the properties of slant and semi-slant submanifolds in metallic (or Golden) Riemannian manifolds.
We have found a relation between the slant angles $\theta$ of a submanifold $M$ in a Riemannian manifold $(\overline{M}, \overline{g})$ endowed with a metallic (or Golden) structure $J$ and the slant angle $\vartheta$ of the same submanifold $M$ of the almost product Riemannian manifold $(\overline{M}, \overline{g}, F)$. Moreover, we have found some integrability conditions for the distributions which are involved in such types of submanifolds in metallic and Golden Riemannian manifolds. We have also given some examples of semi-slant submanifolds in metallic and Golden Riemannian manifolds.

\section{Preliminaries}

First of all we review some basic formulas and definitions for the metallic and Golden structures defined on a Riemannian manifold.

Let $\overline{M}$ be an $m$-dimensional manifold endowed with a tensor field $J$ of type $(1,1)$. We say that the structure $J$ is a \textit{metallic structure} if it verifies:
\begin{equation}\label{e1}
J^{2}= pJ+qI,
\end{equation}
for $p$, $q\in\mathbb{N}^*$, where $I$ is the identity operator on the Lie algebra $\Gamma(T\overline{M})$ of
vector fields on $\overline{M}$. In this situation, the pair $(\overline{M},J)$ is called \textit{metallic manifold}.

If $p=q=1$ one obtains the \textit{Golden structure} (\cite{CrHr}) determined by a $(1,1)$-tensor field $J$  which verifies $J^{2}= J + I$. In this case, $(\overline{M},J)$ is called {\it Golden manifold}.

Moreover, if ($\overline{M}, \overline{g})$ is a Riemannian manifold endowed with a metallic (or a Golden) structure $J$, such that the Riemannian metric $\overline{g}$ is $J$-compatible, i.e.:
\begin{equation} \label{e2}
\overline{g}(JX, Y)= \overline{g}(X, JY),
\end{equation}
for any $X, Y \in \Gamma(T\overline{M})$, then $(\overline{g},J)$ is called a {\it metallic} (or a {\it Golden) Riemannian structure}  and $(\overline{M},\overline{g},J)$ is a {\it metallic (or a Golden) Riemannian manifold}.

We can remark that:
\begin{equation} \label{e3}
\overline{g}(JX, JY)=\overline{g}(J^{2}X, Y) =p \overline{g}(JX,Y)+q \overline{g}(X,Y),
\end{equation}
for any $X, Y \in \Gamma(T\overline{M})$.
\normalfont

Any metallic structure $J$ on $\overline{M}$ induces two almost product structures on this manifold (\cite{Hr4}):
\begin{equation}\label{e4}
F_{1}= \frac{2}{2\sigma _{p, q}-p}J-\frac{p}{2\sigma _{p, q}-p}I, \quad
F_{2}= -\frac{2}{2\sigma _{p, q}-p}J+\frac{p}{2\sigma _{p, q}-p}I.
\end{equation}

Conversely, any almost product structure $F$ on $\overline{M}$ induces two metallic structures on $\overline{M}$ (\cite{Hr4}):
\begin{equation}\label{e5}
(i) \: J_{1}= \frac{2\sigma _{p, q}-p}{2}F+\frac{p}{2}I, \quad
(ii) \: J_{2}=-\frac{2\sigma _{p, q}-p}{2}F+\frac{p}{2}I.
\end{equation}

If the almost product structure $F$ is a Riemannian one, then $J_1$ and $J_2$ are also metallic Riemannian structures.

On a metallic manifold $(\overline{M}, J)$ there are two complementary distributions $\mathcal{D}_{1}$ and $\mathcal{D}_{2}$
corresponding to the projection operators $P$ and $Q$ (\cite{Hr4}), given by:
\begin{equation}\label{e6}
P=-\frac{1}{2\sigma _{p, q}-p} J+\frac{\sigma _{p, q}}{2\sigma _{p, q}-p} I, \quad Q= \frac{1}{2\sigma _{p, q}-p} J+\frac{\sigma _{p, q}-p}{2\sigma _{p, q}-p} I
\end{equation}
and the operators $P$ and $Q$ verify the following relations:
\begin{equation}\label{e7}
P+Q=I, \quad P^{2}=P, \quad Q^{2}=Q, \quad PQ = QP = 0
\end{equation}
and
\begin{equation}\label{e8}
J P=P J=(p-\sigma_{p, q})P, \ \ J Q=Q J=\sigma_{p, q} Q.
\end{equation}

In particular, if $p=q=1$, we obtain that every Golden structure $J$ on $\overline{M}$ induces two almost product structures on this manifold and conversely, an almost product structure $F$ on $\overline{M}$ induces two Golden structures on $\overline{M}$ (\cite{CrHr},\cite{Hr3}).

\section{Submanifolds of metallic Riemannian manifolds}
In the next issues we assume that $M$ is an $m'$-dimensional submanifold, isometrically immersed in the $m$-dimensional metallic (or Golden) Riemannian manifold ($\overline{M}, \overline{g},J)$ with $m, m' \in \mathbb{N}^{*}$ and $m > m'$.

We denote by $T_{x}M$ the tangent space of $M$ in a point $x \in M$ and by $T_{x}^{\bot }M$ the normal space of $M$ in $x$. The tangent space $T_x\overline{M}$ of $\overline{M}$ can be decomposed into the direct sum:
$$T_x\overline{M}=T_xM\oplus T_x^{\perp}M,$$ for any $x\in M$.

Let $i_{*}$ be the differential of the immersion $i: M \rightarrow\overline{M}$. The induced Riemannian metric $g$ on $M$ is given by $g(X, Y)=\overline{g}(i_{*}X, i_{*}Y)$, for any $X, Y \in \Gamma(TM)$. For the simplification of the notations, in the rest of the paper we shall note by $X$ the vector field $i_{*}X$, for any $X \in \Gamma(TM)$.

We consider the decomposition into the tangential and normal parts of $JX$ and $JV$ given by:
\begin{equation}\label{e13}
(i) \: JX = TX + NX, \quad (ii) \: JV = tV + nV,
\end{equation}
for any $X \in \Gamma(TM)$ and $V \in \Gamma(T^{\perp}M)$, where $T:\Gamma(TM)\rightarrow \Gamma(TM)$, $N:\Gamma(TM)\rightarrow \Gamma(T^{\perp}M)$, $t:\Gamma(T^{\perp}M)\rightarrow \Gamma(TM)$, $ n:\Gamma(T^{\perp}M)\rightarrow \Gamma(T^{\perp}M)$, with:
\begin{equation}\label{e14}
  TX:=(J X)^T, \quad NX:=(J X)^{\perp}, \quad  tV:=(J V)^T, \quad nV:=(J V)^{\perp}.
\end{equation}

We remark that the maps $T$ and $n$ are $\overline{g}$-symmetric (\cite{Blaga_Hr}):
\begin{equation}\label{e15}
(i)\: \overline{g}(TX,Y)=\overline{g}(X,TY), \quad (ii)\: \overline{g}(nU,V)=\overline{g}(U,nV)
\end{equation}
and
\begin{equation}\label{e16}
\overline{g}(NX,U)=\overline{g}(X,tU),
\end{equation}
for any $X, Y\in \Gamma(TM)$ and $U, V\in \Gamma(T^{\perp}M)$.

For an almost product structure $F$, the decompositons into tangential and normal parts of $FX$ and $FV$ are given by (\cite{Sahin}):
\begin{equation}\label{e17}
(i) \: FX = fX + \omega X, \quad (ii) \: FV = BV + CV,
\end{equation}
for any $X \in \Gamma(TM)$ and $V \in \Gamma(T^{\perp}M)$, where: $f:\Gamma(TM)\rightarrow \Gamma(TM)$, $\omega :\Gamma(TM)\rightarrow \Gamma(T^{\perp}M)$,
$B :\Gamma(T^{\perp}M)\rightarrow \Gamma(TM)$, $C:\Gamma(T^{\perp}M)\rightarrow \Gamma(T^{\perp}M)$, with:
\begin{equation}\label{e150}
fX:=(F X)^T,\quad \omega X:=(FX)^{\perp}, \quad BV:=(F V)^T, \quad CV:=(F V)^{\perp}.
\end{equation}

The maps $f$ and $C$ are $\overline{g}$-symmetric (\cite{Li&Liu}):
\begin{equation}\label{e18}
\overline{g}(fX,Y)=\overline{g}(X,fY), \ \ \overline{g}(CU,V)=\overline{g}(U,CV)
\end{equation}
 and
\begin{equation}\label{e19}
\overline{g}(\omega X,V)=\overline{g}(X,B V),
\end{equation}
for any $X, Y\in \Gamma(TM)$ and $U, V\in \Gamma(T^{\perp}M)$.

\begin{remark}
Let $(\overline{M}, \overline{g})$ be a Riemannian manifold endowed with an almost product structure $F$ and let $J$ be the metallic structure induced by $F$ on $\overline{M}$. If $M$ is a submanifold in the almost product Riemannian manifold $(\overline{M}, \overline{g}, F)$, then:
\begin{equation} \label{e20}
(i) \: TX = \frac{p}{2}X \pm \frac{2\sigma _{p, q}-p}{2}fX, \quad (ii) \: NX= \pm \frac{2\sigma _{p, q}-p}{2}\omega X
\end{equation}
and
\begin{equation} \label{e21}
(i) \: tV =  \pm \frac{2\sigma _{p, q}-p}{2}BV, \quad (ii) \: nV= \frac{p}{2}V \pm \frac{2\sigma _{p, q}-p}{2}CV,
\end{equation}
for any $X \in \Gamma(TM)$ and $V \in \Gamma(T^{\bot}M)$.
\end{remark}

\begin{remark}
Let $(\overline{M}, \overline{g})$ be a Riemannian manifold endowed with an almost product structure $F$ and let $J$ be the Golden structure induced by $F$ on $\overline{M}$. If $M$ is a submanifold in the almost product Riemannian manifold $(\overline{M}, \overline{g}, F)$, then:
\begin{equation} \label{e22}
(i) \: TX = \frac{1}{2}X \pm \frac{2\phi-1}{2}fX, \quad (ii) \: NX= \pm \frac{2\phi-1}{2}\omega X
\end{equation}
and
\begin{equation} \label{e23}
(i) \: tV = \pm \frac{2\phi-1}{2}BV, \quad (ii) \: nV= \frac{1}{2}V \pm \frac{2\phi-1}{2}CV,
\end{equation}
for any $X \in \Gamma(TM)$ and $V \in \Gamma(T^{\bot}M)$.
\end{remark}

Let $r=m-m'$ be the codimension of $M$ in $\overline{M}$ (where $r,m,m' \in \mathbb{N}^{*}$). We fixe a local orthonormal basis $\{N_{1},...,N_{r}\}$ of the normal space $T_{x}^{\bot}M$. Hereafter we assume that the indices $\alpha, \beta, \gamma $ run over the range $\{1,..., r\}$.

For any $x \in M$ and $X \in T_{x}M$, the vector fields $J(i_{*}X)$ and $JN_{\alpha}$ can be decomposed into  tangential and normal components (\cite{Hr4}):
\begin{equation}\label{e24}
(i)\: JX=TX+\sum_{\alpha=1}^{r}u_{\alpha}(X)N_{\alpha}, \:
(ii)\: JN_{\alpha}=\xi _{\alpha}+\sum_{\beta=1}^{r}a_{\alpha \beta} N_{\beta},
\end{equation}
where $(\alpha \in \{1,..., r\})$, $T$ is an $(1, 1)$-tensor field on $M$, $\xi _{\alpha}$ are vector fields on $M$, $u_{\alpha}$ are $1$-forms on $M$ and $(a_{\alpha\beta})_{r}$ is an $r \times r$ matrix of smooth real functions on $M$.

Using equations (\ref{e13}) and (\ref{e24}), we remark that
\begin{equation}\label{e25}
(i)\: NX=\sum_{\alpha=1}^{r}u_{\alpha}(X)N_{\alpha},\quad (ii)\: tN_{\alpha}=\xi _{\alpha},  \quad (iii)\: nN_{\alpha}=\sum_{\beta=1}^{r}a_{\alpha \beta} N_{\beta}.
\end{equation}

\begin{theorem} \label{t1}
The structure $\Sigma =(T, g, u_{\alpha}, \xi _{\alpha},(a_{\alpha\beta})_{r})$ induced on the submanifold $M$ by the metallic Riemannian structure $(\overline{g}, J)$ on $\overline{M}$ satisfies the following equalities (\cite{Hr5}):
\begin{equation}\label{e26}
     T^{2}X=pTX+qX-\sum_{\alpha=1}^{r}u_{\alpha}(X)\xi _{\alpha},
\end{equation}
\begin{equation}\label{e27}
  (i) \ \ u_{\alpha}(TX)=pu_{\alpha}(X)-\sum_{\beta=1}^{r}a_{\alpha\beta}u_{\beta}(X), \quad
  (ii)\  a_{\alpha\beta}=a_{\beta \alpha}, \quad
  \end{equation}
\begin{equation}\label{e28}
  (i) \ \ u_{\beta}(\xi _{\alpha})=q\delta_{\alpha \beta}+pa_{\alpha\beta}-\sum_{\gamma =1}^{r}a_{\alpha \gamma }a_{\gamma \beta}, \quad
  (ii) \  T\xi _{\alpha}=p\xi _{\alpha}-\sum_{\beta=1}^{r}a_{\alpha\beta}\xi _{\beta},\\
\end{equation}
\begin{equation} \label{e29}
      u_{\alpha}(X)=g(X, \xi_{\alpha})
\end{equation}
for any $X\in \Gamma(TM)$, where $\delta_{\alpha \beta}$ is the Kronecker delta and $p$, $q$ are positive integers \normalfont (\cite{Hr4}).
\end{theorem}

A structure $\Sigma =(T, g, u_{\alpha }, \xi _{\alpha }, (a_{\alpha \beta})_{r})$ induced on the submanifold $M$ by the metallic Riemannian structure $(\overline{g}, J)$ defined on $\overline{M}$ (determined by the $(1, 1)$-tensor field $T$ on $M$, the vector fields $\xi_{\alpha }$ on $M$, the $1$-forms $u_{\alpha }$ on $M$ and the $r \times r$ matrix $(a_{\alpha \beta })_{r}$ of smooth real functions on $M$) which verifies the relations (\ref{e26}),(\ref{e27}),(\ref{e28}) and (\ref{e29}) is called \textit{$\Sigma$ - metallic Riemannian structure} (\cite{Hr5}).

For $p=q=1$, the structure $\Sigma =(T, g, u_{\alpha }, \xi _{\alpha }, (a_{\alpha \beta})_{r})$ is called \textit{$\Sigma$ - Golden Riemannian structure}.

\begin{remark}\label{r1}
If $\Sigma =(T, g, u_{\alpha}, \xi _{\alpha},(a_{\alpha\beta})_{r})$ is the induced structure on the sub\-ma\-ni\-fold $M$ by the metallic (or Golden) Riemannian structure $(\overline{g}, J)$ on $\overline{M}$, then $M$ is an invariant submanifold with respect to $J$  if and only if $(M,T,g)$ is a metallic (or Golden) Riemannian manifold, whenever $T$ is non-trivial (\cite{Hr4}).
\end{remark}

Let $\overline{\nabla}$ and $\nabla $ be the Levi-Civita connections on $(\overline{M},\overline{g})$ and $(M,g)$, respectively. The Gauss and Weingarten formulas are given by:
\begin{equation}\label{e30}
(i) \: \overline{\nabla}_{X}Y=\nabla_{X}Y+h(X,Y), \quad  (ii) \: \overline{\nabla}_{X}V=-A_{V}X+\nabla_{X}^{\bot}V,
\end{equation}
for any $X, Y \in \Gamma(TM)$ and $V \in \Gamma(T^{\bot}M)$, where $h$ is the second fundamental form, $A_{V}$ is the shape operator of $M$ and the second fundamental form $h$ and the shape operator $A_{V}$ are related by:
 \begin{equation}\label{e31}
 \overline{g}(h(X, Y),V)=\overline{g}(A_{V}X, Y).
 \end{equation}

\begin{remark}
Using a local orthonormal basis $\{N_{1},...,N_{r}\}$ of the normal space $T_{x}^{\bot}M$, where $r$ is the codimension of $M$ in $\overline{M}$ and $A_{\alpha}:=A_{N_{\alpha}}$, for any $\alpha \in \{1,...,r\}$, we obtain:
\begin{equation}\label{e32}
(i) \: \overline{\nabla}_{X}N_{\alpha}=-A_{\alpha}X+\nabla_{X}^{\bot}N_{\alpha}, \quad (ii) \: h_{\alpha}(X, Y)=g(A_{\alpha}X, Y),
\end{equation}
for any $X, Y \in \Gamma(TM)$.
\end{remark}

\begin{remark}\label{r2}
For any $\alpha \in \{1,..., r\}$, the normal connection $\nabla_{X}^{\bot}N_{\alpha}$ has the decomposition
$ \nabla_{X}^{\bot}N_{\alpha}=\sum_{\beta=1}^{r}l_{\alpha\beta}(X)N_{\beta},$
for any $X \in \Gamma(TM)$, where $(l_{\alpha\beta})_{r}$ is an $r\times r$ matrix of $1$-forms on $M$.
Moreover, from $\overline{g}(N_{\alpha}, N_{\beta})= \delta_{\alpha\beta}$, we obtain (\cite{Hr5}):
$
 \overline{g}(\nabla_{X}^{\bot}N_{\alpha}, N_{\beta})+\overline{g}(N_{\alpha}, \nabla_{X}^{\bot}N_{\beta})=0,
$
 which is equivalent to $ l_{\alpha\beta}=-l_{\beta\alpha}$,  for any
 $\alpha, \beta \in \{1,..., r\}$ and any $X\in \Gamma (TM)$.
\end{remark}

The covariant derivatives of the tangential and normal parts of $JX$ (and $JV$), $T$ and $N$ ($t$ and $n$, respectively) are given by:
\begin{equation}\label{e33}
(i) \: (\nabla_{X}T)Y=\nabla_{X}TY - T(\nabla_{X}Y), \quad (ii) \:(\overline{\nabla}_{X}N)Y=\nabla_{X}^{\bot}NY - N(\nabla_{X}Y),
\end{equation}
and
\begin{equation}\label{e330}
(i) \: (\nabla_{X}t)V=\nabla_{X}tV - t(\nabla_{X}^{\bot}V), \quad (ii) \:(\overline{\nabla}_{X}n)V=\nabla_{X}^{\bot}nV - n(\nabla_{X}^{\bot}V),
\end{equation}
for any $X$, $Y \in \Gamma(TM)$ and $V \in \Gamma(T^{\bot}M) $. From $\overline{g}(JX,Y)=\overline{g}(X,JY)$, it follows:
\begin{equation} \label{e34}
\overline{g}((\overline{\nabla}_XJ)Y,Z)=\overline{g}(Y,(\overline{\nabla}_XJ)Z),
\end{equation}
for any $X$, $Y$, $Z\in \Gamma(T\overline{M})$. Moreover, if $M$ is an isometrically immersed submanifold of the metallic Riemannian manifold $(\overline{M},\overline{g},J)$, then (\cite{Blaga2}):
\begin{equation}\label{e35}
\overline{g}((\nabla_X T)Y,Z)=\overline{g}(Y,(\nabla_X T)Z),
\end{equation}
for any $X$, $Y$, $Z\in \Gamma(TM)$.

\bigskip

Using an analogy of a locally product manifold (\cite{Pitis}), we can define \textit{locally metallic (or locally Golden) Riemannian manifold}, as follows (\cite{Hr5}):

\begin{definition}\label{d2}
If $(\overline{M},\overline{g}, J)$ is a metallic (or Golden) Riemannian manifold and $J$ is parallel with respect to the Levi-Civita connection $\overline{\nabla}$ on $\overline{M}$ (i.e. $\overline{\nabla}J=0$), we say that  $(\overline{M},\overline{g}, J)$ is a {\it locally metallic (or locally Golden) Riemannian manifold}.
\end{definition}

\begin{prop}
If $M$ is a submanifold of a locally metallic (or locally Golden) Riemannian manifold $(\overline{M},\overline{g},J)$, then:
\begin{equation}\label{e36}
T([X,Y])=\nabla_{X}TY-\nabla_{Y}TX-A_{NY}X+A_{NX}Y
\end{equation}
and
\begin{equation}\label{e37}
N([X,Y])=h(X,TY)-h(TX,Y)+\nabla_{X}^{\bot}NY-\nabla_{Y}^{\bot}NX,
\end{equation}
for any $X, Y \in \Gamma(TM)$, where $\nabla$ is the Levi-Civita connection on $M$.
\end{prop}

 \begin{proof}
From $(\overline{M},\overline{g},J)$ locally metallic (or locally Golden) Riemannian manifold, we have $(\overline{\nabla}_{X}J)Y=0$, for any $X, Y \in \Gamma(TM)$.

Thus:
$\overline{\nabla}_{X}(TY+NY)=J(\nabla_{X}Y+h(X,Y))$, that is equivalent to:
$$ \nabla_{X}TY+h(X,TY)-A_{NY}X+\nabla_{X}^{\bot}NY=T(\nabla_{X}Y)+N(\nabla_{X}Y)+th(X,Y)+nh(X,Y).$$
Taking the normal and the the tangential components of this equality, we get:
\begin{equation}\label{e38}
N(\nabla_{X}Y)-\nabla_{X}^{\bot}NY=h(X,TY)-nh(X,Y)
\end{equation}
and
\begin{equation}\label{e39}
\nabla_{X}TY - T (\nabla_{X}Y)=A_{NY}X+th(X,Y).
\end{equation}

Interchanging $X$ and $Y$ and subtracting these equalities, we obtain the tangential and normal components of $[X,Y]=\nabla_{X}Y-\nabla_{Y}X$, which give us (\ref{e36}) and (\ref{e37}).
\end{proof}

From (\ref{e33}) and (\ref{e330}) we obtain:
\begin{prop}
If $M$ is a submanifold of a locally metallic (or Golden) Riemannian manifold $(\overline{M},\overline{g},J)$, then the covariant derivatives of $T$ and $N$ verify:
\begin{equation}\label{e40}
(i)  (\nabla_{X}T)Y=A_{NY}X+th(X,Y), \: (ii)  (\overline{\nabla}_{X}N)Y=nh(X,Y)-h(X,TY),
\end{equation}
and
\begin{equation}\label{e400}
(i)  (\nabla_{X}t)V=A_{nV}X - TA_{V}X, \: (ii)  (\overline{\nabla}_{X}n)V=-h(X,tV)-NA_{V}X,
\end{equation}
for any $X$, $Y \in \Gamma(TM)$ and $V \in \Gamma(T^{\bot}M)$.
\end{prop}

\begin{prop} \label{t2}
 If $M$ is an $n$-dimensional submanifold of codimension $r$ in a locally metallic (or locally Golden) Riemannian manifold $(\overline{M},\overline{g}, J)$, then the structure  $\Sigma =(T, g, u_{\alpha}, \xi _{\alpha},(a_{\alpha\beta})_{r})$ induced on $M$ by the metallic (or Golden) Riemannian structure $(\overline{g}, J)$, has the following properties (\cite{Hr5}):

\begin{equation} \label{e41}
(\nabla_{X}T)Y=\sum_{\alpha=1}^{r}h_{\alpha}(X, Y)\xi _{\alpha}+\sum_{\alpha=1}^{r}u_{\alpha}(Y)A_{\alpha}X,
\end{equation}

\begin{equation}\label{e42}
(\nabla_{X}u_{\alpha})Y=-h_{\alpha}(X, TY)+\sum_{\beta=1}^{r}[u_{\beta}(Y)l_{\alpha\beta}(X)+h_{\beta}(X, Y)a_{\beta\alpha}],
\end{equation}
for any $X, Y \in \Gamma(TM)$.
\end{prop}

\begin{proof}
From $\overline{\nabla} J = 0 $ we obtain
$\overline{\nabla}_{X} J Y = J(\overline{\nabla}_{X}Y)$,
for any $X,Y \in \Gamma(T\overline{M})$.
Using (\ref{e30})(i), (\ref{e32}) and (\ref{e24})(ii), we get:
\[ \overline{\nabla}_{X}J Y = \overline{\nabla}_{X}TY+
\sum_{\alpha=1}^{r} X(u_{\alpha}(Y))N_{\alpha}+\sum_{\alpha=1}^{r} u_{\alpha}(Y)\overline{\nabla}_{X}N_{\alpha}=\]
\[=\nabla_{X}TY-\sum_{\alpha=1}^{r} u_{\alpha}(Y)A_{\alpha}X+\sum_{\alpha=1}^{r} [h_{\alpha}(X,TY)+X(u_{\alpha}(Y))
+\sum_{\beta=1}^{r} u_{\beta}(Y)l_{\beta\alpha}(X)]N_{\alpha}\]
and
\[ J (\overline{\nabla}_{X}Y)=J (\nabla_{X}Y)+\sum_{\alpha=1}^{r} h_{\alpha}(X,Y)J N_{\alpha}=\]
\[=T(\nabla_{X}Y)+ \sum_{\alpha=1}^{r} h_{\alpha}(X,Y)\xi_{\alpha}+\sum_{\alpha=1}^{r} [u_{\alpha}(\nabla_{X}Y)+\sum_{\beta=1}^{r} h_{\beta}(X,Y)a_{\beta\alpha}]N_{\alpha} ,\]
for any $X,Y \in \Gamma(TM)$.

Identifying the tangential and normal components, respectively, of the last two equalities, we get (\ref{e41}) and (\ref{e42}).
\end{proof}

Using (\ref{e36}), (\ref{e37}), (\ref{e41}) and (\ref{e42}), we obtain:
\begin{prop}
If $M$ is a submanifold of a locally metallic (or Golden) Riemannian manifold $(\overline{M},\overline{g},J)$, then:
\begin{equation}\label{e43}
T([X,Y])=\nabla_{X}TY-\nabla_{Y}TX-\sum_{i=1}^{r}[u_{\alpha}(Y)A_{\alpha}X-u_{\alpha}(X)A_{\alpha}Y]
\end{equation}
and
\begin{equation}\label{e44}
N([X,Y])=\sum_{\alpha=1}^{n}[((\nabla_{Y}u_{\alpha})X-(\nabla_{X}u_{\alpha})Y)+
(u_{\alpha}(X)l_{\alpha\beta}(Y)-u_{\alpha}(Y)l_{\alpha\beta}(X))]N_{\alpha},
\end{equation}
for any $X, Y \in \Gamma(TM)$, where $\nabla$ is the Levi-Civita connection on $M$.
\end{prop}

\section{Slant submanifolds in metallic or Golden Riemannian manifolds}

Let $M$ be an $m'$-dimensional submanifold, isometrically immersed in an $m$-dimensional metallic (or Golden) Riemannian manifold ($\overline{M}, \overline{g}, J)$, where $m,m' \in \mathbb{N}^{*}$ and $m > m'$. Using the Cauchy-Schwartz inequality (\cite{Alegre}), we have:
$$ \overline{g}(JX,TX) \leq \| JX\| \cdot \| TX\|,$$
for any $X \in \Gamma(TM)$. Thus, there exists a function $\theta:\Gamma(TM)\rightarrow [0,\pi]$, such that:
$$\overline{g}(JX_x,TX_x)=\cos\theta(X_x) \| TX_x\| \cdot \| JX_x \|,$$ for any $x\in M$ and any nonzero tangent vector $X_x\in T_xM$. The angle $\theta(X_x)$ between $JX_x$ and $T_xM$ is called the \textit{Wirtinger angle} of $X$ and it verifies:
$$
\cos\theta(X_x) =\frac{\overline{g}(JX_x,TX_x)}{\| TX_x\| \cdot \| JX_x \|}.
$$

\begin{definition}\label{d3}(\cite{Blaga_Hr})
A submanifold $M$ in a metallic (or Golden) Riemannian manifold ($\overline{M}, \overline{g}, J)$ is called \textit{slant submanifold} if the angle $\theta(X_x)$ between $JX_x$ and $T_xM$ is constant, for any $x\in M$ and $X_x\in T_xM$. In such a case, $\theta=:\theta(X_x)$ is called the \textit{slant angle} of $M$ in $\overline{M}$, and it verifies:
 \begin{equation}\label{e45}
\cos\theta =\frac{\overline{g}(JX,TX)}{\| JX \| \cdot \| TX \|}=\frac{\| TX \|}{\| JX \|}.
\end{equation}
 The immersion $i: M \rightarrow \overline{M} $ is named {\it slant immersion} of $M$ in $\overline{M}$.
\end{definition}
\normalfont

\begin{remark}\label{r3}
 The invariant and anti-invariant submanifolds in the metallic (or Golden) Riemannian manifold ($\overline{M}, \overline{g}, J)$ are particular cases of slant submanifolds with the slant angle $\theta=0$ and $\theta=\frac{\pi}{2}$, respectively. A slant submanifold $M$ in $\overline{M}$, which is neither invariant nor anti-invariant, is called  \textit{proper slant submanifold} and the immersion $i: M \rightarrow \overline{M} $ is called \textit{proper slant immersion}.
\end{remark}

Taking $X+Y$ in $\overline{g}(TX,TX)=\cos^2 \theta \overline{g}(JX,JX)$ we obtain:
\begin{prop}\label{p1} (\cite{Blaga_Hr})
Let $M$ be an isometrically immersed submanifold of the metallic Riemannian manifold $(\overline{M}, \overline{g}, J)$. If $M$ is a slant submanifold with the slant angle $\theta$, then:
\begin{equation}\label{e46}
\overline{g}(TX,TY)=\cos^2\theta[p\overline{g}(X,TY)+q \overline{g}(X,Y)]
\end{equation}
and
\begin{equation}\label{e47}
\overline{g}(NX,NY)=\sin^2\theta[p\overline{g}(X,TY)+q\overline{g}(X,Y)],
\end{equation}
for any $X$, $Y\in \Gamma(TM)$.

Moreover, we have
\begin{equation}\label{e48}
T^2=\cos^2\theta(pT+qI),
\end{equation}
where $I$ is the identity on $\Gamma(TM)$ and
\begin{equation}\label{e49}
\nabla(T^2)=p\cos^2\theta(\nabla T).
\end{equation}
\end{prop}

\begin{remark}
From (\ref{e26}) and (\ref{e48}), we have:
\begin{equation}\label{e50}
\sin^2 \theta (pT+qI) =\sum_{\alpha=1}^{r}u_{\alpha}\otimes \xi _{\alpha}.
\end{equation}
where $I$ is the identity on $\Gamma(TM)$.
\end{remark}

\begin{prop}
If $M$ is an isometrically immersed slant submanifold  of the Golden Riemannian manifold $(\overline{M}, \overline{g}, J)$ with the slant angle $\theta$, then:
\begin{equation}\label{e51}
\overline{g}(TX,TY)=\cos^2\theta[\overline{g}(X,TY)+ \overline{g}(X,Y)]
\end{equation}
and
\begin{equation}\label{e52}
\overline{g}(NX,NY)=\sin^2\theta[\overline{g}(X,TY)+\overline{g}(X,Y)],
\end{equation}
for any $X$, $Y\in \Gamma(TM)$. Moreover, we have
\begin{equation}\label{e53}
T^2=\cos^2\theta(T+I),\quad \nabla(T^2)=\cos^2\theta(\nabla T),
\end{equation}
 and
\begin{equation}\label{e54}
\sin^2 \theta (T+I) =\sum_{\alpha=1}^{r}u_{\alpha}\otimes \xi _{\alpha},
\end{equation}
where $I$ is the identity on $\Gamma(TM)$.
\end{prop}

\begin{definition}\label{d33}(\cite{Atceken1})
The submanifold $M$ in the almost product Riemannian manifold ($\overline{M}, \overline{g}, F)$ is a \textit{slant submanifold} if the angle $\vartheta(X_x)$ between $JX_x$ and $T_xM$ is constant, for any $x\in M$ and $X_x\in T_xM$. In such a case, $\vartheta=:\vartheta(X_x)$ is called the \textit{slant angle} of the submanifold $M$ in $\overline{M}$ and it verifies:
 \begin{equation}\label{e55}
\cos\vartheta =\frac{\overline{g}(FX,fX)}{\| FX \| \cdot \| fX \|}=\frac{\| fX \|}{\| FX \|}.
\end{equation}
\end{definition}

\begin{prop}\label{p33}(\cite{Li&Liu})
If $M$ is a slant submanifold isometrically immersed in the almost product Riemannian manifold $(\overline{M}, \overline{g}, F)$ with the slant angle $\vartheta$, then:
\begin{equation}\label{e56}
(i)\: \overline{g}(fX,fY)=\cos^2\vartheta \overline{g}(X,Y), \:(ii)\:\overline{g}(\omega X,\omega Y)=\sin^2\vartheta \overline{g}(X,Y),
\end{equation}
for any $X, Y\in \Gamma(TM)$.
\end{prop}

In the next proposition we find a relation between the slant angles $\theta$ of the submanifold $M$ in the metallic Riemannian manifold $(\overline{M}, \overline{g}, J)$ and the slant angle $\vartheta$ of the submanifold $M$ in the almost product Riemannian manifold $(\overline{M}, \overline{g}, F)$.

\begin{theorem}
Let $M$ be a submanifold in the Riemannian manifold $(\overline{M}, \overline{g}$) endowed with an almost product structure $F$ on $\overline{M}$ and let $J$ be the induced metallic structure by $F$ on $(\overline{M}, \overline{g}$). If $M$ is a slant submanifold in the almost product Riemannian manifold $(\overline{M}, \overline{g}, F)$ with the slant angle $\vartheta$ and $F \neq -I$ ($I$ is the identity on $\Gamma(TM)$) and $J=\frac{2\sigma _{p, q}-p}{2}F+\frac{p}{2}I$, then $M$ is a slant submanifold in the metallic Riemannian manifold $(\overline{M}, \overline{g}, J)$ with slant angle $\theta$ given by:
\begin{equation}\label{e57}
 \sin\theta=\frac{2\sigma_{p,q} - p}{2\sigma_{p,q}}\sin\vartheta.
\end{equation}
\end{theorem}

\begin{proof}
From (\ref{e20}(ii)), we obtain:
$\overline{g}(NX,NY)= \frac{(2\sigma _{p, q}-p)^{2}}{4}\overline{g}(\omega X,\omega Y)$, for any $X, Y\in \Gamma(TM)$.
From (\ref{e47}) and (\ref{e56}(ii)), we get:
\begin{equation}\label{e58}
\frac{(2\sigma _{p, q}-p)^{2}}{4}\overline{g}(X,Y)\sin^{2}\vartheta = [p\overline{g}(X,JY) + q \overline{g}(X,Y)]\sin^{2}\theta,
\end{equation}
for any $X, Y\in \Gamma(TM)$.
Using $J=\frac{p}{2}I +\frac{2\sigma _{p, q}-p}{2}F$, we have:
\begin{equation}\label{e59}
(2\sigma_{p, q}-p)^{2}\overline{g}(X,Y)\sin^{2}\vartheta = [(2p^{2}+4q)\overline{g}(X,Y) + 2p\sqrt{p^{2}+4q}\overline{g}(X,FY)]\sin^{2}\theta,
\end{equation}
for any $X, Y\in \Gamma(TM)$. Replacing $Y$ by $FY$ and using $F^{2}Y=Y$, for any $Y \in \Gamma(TM)$, we obtain:
\begin{equation}\label{e60}
(2\sigma_{p, q}-p)^{2}\overline{g}(X,FY)\sin^{2}\vartheta = [(2p^{2}+4q)\overline{g}(X,FY) + 2p\sqrt{p^{2}+4q}\overline{g}(X,Y)]\sin^{2}\theta.
\end{equation}
for any $X, Y\in \Gamma(TM)$.

Summing the equalities (\ref{e59}) and (\ref{e60}), we obtain:
\begin{equation}\label{e61}
\overline{g}(X,FY+Y)[(2\sigma_{p,q} - p)^{2}\sin^{2}\vartheta -4(q+p\sigma_{p,q})\sin^{2}\theta]=0,
\end{equation}
for any $X, Y\in \Gamma(TM)$.

Using $q+p\sigma_{p,q}=\sigma_{p,q}^{2}$, $FY \neq -Y$ and $\theta, \vartheta  \in [0, \pi )$ in (\ref{e61}), we get  (\ref{e57}).

\end{proof}

In particular, for $p=q=1$, we obtain the relation between slant angle $\theta$ of the immersed submanifold $M$ in a Golden Riemannian manifold $(\overline{M}, \overline{g},J$) and the slant angle $\vartheta$ of $M$ immersed in the almost product Riemannian manifold $(\overline{M}, \overline{g},F)$:

\begin{prop}
Let $M$ be a submanifold in the Riemannian manifold $(\overline{M}, \overline{g}$) endowed with an almost product structure $F$ on $\overline{M}$ and let $J$ be the induced Golden structure by $F$ on $(\overline{M}, \overline{g}$). If $M$ is a slant submanifold in the almost product Riemannian manifold $(\overline{M}, \overline{g}, F)$ with the slant angle $\vartheta$ and $F \neq -I$ ($I$ is the identity on $\Gamma(TM)$) and $J=\frac{2\phi-1}{2}F+\frac{1}{2}I$, then  $M$ is a slant submanifold in the Golden Riemannian manifold $(\overline{M}, \overline{g}, J)$ with slant angle $\theta$ given by:
\begin{equation}\label{e62}
\sin\theta=\frac{2\phi -1}{2\phi}\sin\vartheta,
\end{equation}
where $\phi=\frac{1+\sqrt{5}}{2}$ is the Golden number.
\end{prop}

\section{Semi-slant submanifolds in metallic or Golden Riemannian manifolds}

\normalfont

We define the slant distribution of a metallic (or Golden) Riemannian manifold, using a similar definition as for Riemannian product manifold (\cite{Li&Liu},\cite{Sahin}).

\begin{definition}\label{d4}
Let $M$ be an immersed submanifold of a metallic (or Golden) Riemannian manifold $(\overline{M},\overline{g},J)$. A differentiable distribution $D$ on $M$ is called a {\it slant distribution} if the angle $\theta_{D}$ between $JX_{x}$ and the vector subspace $D_{x}$ is constant, for any $x \in M$  and any nonzero vector field $X_{x} \in \Gamma(D_{x})$. The constant angle $\theta_{D}$ is called the {\it slant angle} of the distribution $D$.
\end{definition}

\begin{prop}
Let $D$ be a differentiable distribution on a submanifold $M$ of a metallic (or Golden) Riemannian manifold $(\overline{M},\overline{g},J)$. The distribution $D$ is a slant distribution if and only if there exists a constant $\lambda \in [0, 1]$ such that:
\begin{equation}\label{e100}
(P_{D}T)^{2}X= \lambda(pP_{D}TX+qX),
\end{equation}
 for any $ X \in \Gamma(D)$, where $P_{D}$ is the orthogonal projection on $D$. Moreover, if $\theta_{D}$ is the slant angle of $D$, then it satisfies $\lambda = \cos^{2} \theta_{D}$.
 \end{prop}

\begin{proof}
If the distribution $D$ is a slant distribution on $M$, by using $$\cos\theta_{D} =\frac{\overline{g}(JX,P_{D}TX)}{\| JX \| \cdot \| P_{D}TX \|}=\frac{\| P_{D}TX \|}{\| JX \|},$$ we get
$\overline{g}(P_{D}TX,P_{D}TX)=\cos^2 \theta_{D} \overline{g}(JX, JX)= \cos^2 \theta_{D} \overline{g}(pP_{D}TX+qX,X),$ for any $X\in \Gamma(D)$ and we obtain (\ref{e100}).

Conversely, if there exists a constant $\lambda \in [0, 1]$ such that (\ref{e100}) holds for any $X \in \Gamma(D)$, we obtain
$\overline{g}(JX,P_{D}TX)=\overline{g}(X,JP_{D}TX)=\overline{g}(X,(P_{D}T)^{2}X)=\lambda\overline{g}(X,pP_{D}TX+qX)=\lambda\overline{g}(X,pJTX+qX)=\lambda\overline{g}(X,J^{2}X)$. Thus,
$$\cos\theta_{D} =\frac{\overline{g}(JX,P_{D}TX)}{\| JX \| \cdot \| P_{D}TX \|}=\lambda\frac{\overline{g}(JX,JX)}{\| JX \| \cdot \| P_{D}TX\|}=\lambda\frac{\| JX \|}{\|P_{D}TX\|},$$
and using $\cos\theta =\frac{\| P_{D}TX \|}{\| JX \|}$ we get  $\cos^{2}\theta_{D} = \lambda$. Thus, $\cos^{2}\theta_{D}$ is constant and $D$ is a slant distribution on $M$.
\end{proof}

\begin{definition}\label{d5}
Let $M$ be an immersed submanifold in a metallic (or Golden) Riemannian manifold $(\overline{M},\overline{g},J)$. We say that $M$ is a {\it bi-slant submanifold} of $\overline{M}$ if there exist two orthogonal differentiable distribution $D_{1}$ and $D_{2}$ on $M$ such that $TM = D_{1}\oplus D_{2}$, and $D_{1}$, $D_{2}$ are slant distribution with the slant angles $\theta_{1}$ and $\theta_{2}$, respectively. Moreover, $M$ is a {\it proper bi-slant submanifold} of $\overline{M}$ if $dim(D_{1})\cdot dim(D_{2}) \neq 0 $.
\end{definition}

For a differentiable distribution $D_{1}$ on $M$,  we denote by $D_{2}:= D_{1}^{\bot}$  the orthogonal distribution of $D_{1}$ in $M$ (i.e. $TM = D_{1} \oplus D_{2} $). Let $P_{1}$ and $P_{2}$ be the orthogonal projections on $D_{1}$ and $D_{2}$. Thus, for any $X \in \Gamma(TM)$, we can consider the decomposition of $X=P_{1}X + P_{2}X$, where $P_{1}X \in \Gamma(D_{1})$ and $P_{2}X \in \Gamma(D_{2})$.

 If $M$ is a bi-slant submanifold of a metallic Riemannian manifold $(\overline{M},\overline{g},J)$ with the orthogonal distribution $D_{1}$ and $D_{2}$ and the slant angles $\theta_{1}$ and $\theta_{2}$ respectively, then
$
JX=P_{1}TX+P_{2}TX+NX = TP_{1}X+ TP_{2} X + N P_{1}X + NP_{2}X,
$
for any $X \in \Gamma(TM)$.

 In a similar manner as in (\cite{Li&Liu}), we can prove:

\begin{prop}
If $M$ is a bi-slant submanifold in a metallic (or Golden) Riemannian manifold $(\overline{M},\overline{g},J)$, with the slant angles $\theta_{1}=\theta_{2}=\theta$ and $g(JX,Y)=0$, for any $X \in \Gamma(D_{1})$ and $Y \in \Gamma(D_{2})$, then $M$ is a slant submanifold in the metallic Riemannian manifold $(\overline{M},\overline{g},J)$ with the slant angle $\theta$.
\end{prop}

\begin{proof}
From $\overline{g}(JX,Y)=\overline{g}(TX,Y)=0$, for any $X \in \Gamma(D_{1})$ and $Y \in \Gamma(D_{2})$, follows $\overline{g}(X,JY)=\overline{g}(X,TY)=0$. Thus, we obtain $TX \in \Gamma(D_{1})$, for any $X \in \Gamma(D_{1})$ and $TY \in \Gamma(D_{2})$, for any $Y \in \Gamma(D_{2})$.

Moreover, using the projections of any $X \in \Gamma(TM)$ on $\Gamma(D_{1})$ and $\Gamma(D_{2})$ respectively, we obtain the decomposition $X=P_{1}X + P_{2}X$, where $P_{1}X \in \Gamma(D_{1})$ and $P_{2}X \in \Gamma(D_{2})$.

From $\overline{g}(TP_{i}X,TP_{i}X)=\cos^{2}\theta_{i}g(JP_{i}X,JP_{i}X)$ (for $i \in \{ 1,2 \}$) and using $\theta_{1}=\theta_{2}=\theta$, we obtain:
$$ \frac{\overline{g}(TX,TX)}{\overline{g}(JX,JX)}=\frac{\overline{g}(TP_{1}X,TP_{1}X)+\overline{g}(TP_{2}X,TP_{2}X)}{\overline{g}(JP_{1}X,JP_{1}X)+\overline{g}(JP_{2}X,JP_{2}X)}=\cos^{2}\theta,$$
for any $X \in \Gamma(TM)$. Thus, $M$ is a slant submanifold in the metallic (or Golden) Riemannian manifold $(\overline{M},\overline{g},J)$ with the slant angle $\theta$.
\end{proof}

If $M$ is a {\it bi-slant submanifold} of a manifold $\overline{M}$, for particular values of the angles $\theta_{1}=0$ and $\theta_{2}\neq 0$, we obtain:

\begin{definition}\label{d6}
An immersed submanifold $M$ in a metallic (or Golden) Riemannian manifold $(\overline{M},\overline{g},J)$ is a semi-slant submanifold if there exist two orthogonal distributions
$D_{1}$ and $D_{2}$ on $M$ such that:

(1) $TM$ admits the orthogonal direct decomposition $TM=D_{1}\oplus D_{2}$;

(2) The distribution $D_{1}$ is invariant distribution (i.e. $J(D_{1})=D_{1}$);

(3) The distribution $D_{2}$ is slant with angle $\theta \neq 0$.

Moreover, if $dim(D_{1})\cdot dim(D_{2}) \neq 0 $, then $M$ is a proper semi-slant submanifold.

\end{definition}

\begin{remark}
If $M$ is a semi-slant submanifold of a metallic Riemannian manifold $(\overline{M},\overline{g},J)$ with the slant angle $\theta$ of the distributions $D_{2}$, then we get:

-  $M$ is an invariant submanifold if $dim(D_{2})=0$;

- $M$ is an anti-invariant submanifold if $dim(D_{1})=0$ and $\theta=\frac{\pi}{2}$;

- $M$ is a semi-invariant submanifold if $D_{2}$ is anti-invariant (i.e.  $\theta = \frac{\pi}{2}$).

\end{remark}

If $M$ is a semi-slant submanifold in a metallic (or Golden) Riemannian manifold $(\overline{M},\overline{g},J)$, then:
\begin{equation}\label{e63}
JX= TP_{1}X+ TP_{2} X + NP_{2}X=P_{1}TX+ P_{2} TX + NP_{2}X
\end{equation}
and
\begin{equation}\label{e64}
(i) JP_{1}X= TP_{1}X,  \: (ii) NP_{1}X=0,   \: (iii) TP_{2} X \in \Gamma(D_{2}),
\end{equation}
for any $X \in \Gamma(TM)$. Moreover, we have
$$\overline{g}(JP_{2}X,TP_{2}X)=\cos \theta(X) \| TP_{2}X\| \cdot \|JP_{2}X\|$$
and the cosine of the slant angle $\theta(X)$ of the distribution $D_{2}$ is constant, for any nonzero $X \in \Gamma(TM)$. If $\theta(X)=:\theta$, we get:
\begin{equation}\label{e65}
\cos \theta =\frac{\overline{g}(JP_{2}X, TP_{2}X)}{\|TP_{2}X\| \cdot \|JP_{2}X\|}=\frac{\|TP_{2}X \|}{\|JP_{2}X\|},
\end{equation}
for any nonzero $X \in \Gamma(TM)$.

\begin{prop}\label{p11}
If $M$ is a semi-slant submanifold of the metallic Riemannian manifold $(\overline{M},\overline{g},J)$ with the slant angle $\theta$ of the distribution $D_{2}$, then:
\begin{equation}\label{e66}
\overline{g}(TP_{2}X,TP_{2}Y)=\cos^2 \theta[p \overline{g}(TP_{2}X,P_{2}Y)+q \overline{g}(P_{2}X,P_{2}Y)]
\end{equation}
and
\begin{equation}\label{e67}
\overline{g}(NX,NY)=\sin^2 \theta[p \overline{g}(TP_{2}X,P_{2}Y)+q \overline{g}(P_{2}X,P_{2}Y)],
\end{equation}
for any $X$, $Y\in \Gamma(TM)$.
\end{prop}
\begin{proof}
Taking $X+Y$ in (\ref{e65}) we have:
$$
\overline{g}(TP_{2}X,TP_{2}Y)=\cos^{2}\theta  \overline{g}(JP_{2}X,JP_{2}Y)= \cos^{2}\theta[p\overline{g}(JP_{2}X,P_{2}Y)+q\overline{g}(P_{2}X,P_{2}Y)],$$
for any $X$, $Y\in \Gamma(TM)$ and using (\ref{e64})(iii) we get (\ref{e66}).

From (\ref{e64})(ii) we get $TP_{2}X=JP_{2}X-NX$, for any $X\in \Gamma(TM)$. Thus, we obtain:
$$\overline{g}(TP_{2}X,TP_{2}Y)=\overline{g}(JP_{2}X,JP_{2}Y)-\overline{g}(NX,NY),$$ for any $X$, $Y\in \Gamma(TM)$ and it implies (\ref{e67}).

\end{proof}

\begin{remark}
A semi-slant submanifold $M$ of a Golden Riemannian manifold $(\overline{M},\overline{g},J)$ with the slant angle $\theta$ of the distribution $D_{2}$ verifies:
\begin{equation}
\overline{g}(TP_{2}X,TP_{2}Y)=\cos^2 \theta[\overline{g}(TP_{2}X,P_{2}Y)+\overline{g}(P_{2}X,P_{2}Y)],
\end{equation}
and
\begin{equation}
\overline{g}(NX,NY)=\sin^2 \theta[\overline{g}(TP_{2}X,P_{2}Y)+\overline{g}(P_{2}X,P_{2}Y)],
\end{equation}
for any $X$, $Y\in \Gamma(TM)$.
\end{remark}

\begin{prop}
Let $M$ be a semi-slant submanifold of a metallic Riemannian manifold $(\overline{M}, \overline{g},J)$ with the slant angle $\theta$ of the distribution $D_{2}$. Then:
\begin{equation}\label{e68}
(TP_{2})^2=\cos^2 \theta(p TP_{2}+qI),
\end{equation}
where $I$ is the identity on $\Gamma(D_{2})$ and
\begin{equation}\label{e69}
\nabla ((TP_{2})^2)=p \cos^2 \theta \nabla (TP_{2}).
\end{equation}
\end{prop}

\begin{proof}
Using $\overline{g}(TP_{2}X,TP_{2}Y)=\overline{g}((TP_{2})^{2}X,P_{2}Y)$, for any $X$, $Y\in \Gamma(TM)$ and (\ref{e66}), we obtain (\ref{e68}).
Moreover we have:
$$(\nabla_X(TP_{2})^2)Y=\cos^2 \theta(p (\nabla_X TP_{2})Y+q (\nabla_X I)Y )=p\cos^2 \theta \nabla_X (P_{2}T)Y,$$ for any $X\in \Gamma(D_{2})$ and $Y\in \Gamma(TM)$. For the identity $I$ on $\Gamma(D_{2})$ we have $(\nabla_X I)P_{2}Y=0$ thus, we get (\ref{e69}).
\end{proof}

\begin{remark}
A semi-slant submanifold $M$ of a Golden Riemannian manifold $(\overline{M}, \overline{g},J)$ with the slant angle $\theta$ of the distribution $D_{2}$ verifies:
\begin{equation}\label{e70}
(TP_{2})^2=\cos^2 \theta(TP_{2}+I),
\end{equation}
where $I$ is the identity on $\Gamma(D_{2})$ and
\begin{equation}\label{e71}
 \nabla ((TP_{2})^2)=\cos^2 \theta \nabla (TP_{2}).
\end{equation}
\end{remark}

\begin{prop}
Let $M$ be an immersed submanifold of a metallic Riemannian manifold $(\overline{M},\overline{g},J)$. Then $M$ is a semi-slant submanifold in $\overline{M}$ if and only if exists a constant $\lambda \in [0, 1)$ such that:
 $$D=\{ X \in \Gamma(TM) | T^{2}X= \lambda(pTX+qX)\}$$
is a distribution and $NX=0$, for any $X \in \Gamma(TM)$ orthogonal to $D$, where $p$ and $q$ are given in (\ref{e1}).
\end{prop}

\begin{proof}
If we consider $M$ a semi-slant submanifold of the metallic Riemannian manifold $(\overline{M},\overline{g},J)$ then, in (\ref{e66})
we put $ \lambda=\cos^2 \theta \in [0,1)$. Thus, we obtain:
$T^2X=\lambda(pTX+qX)$ and we get $D_{2}\subseteq D$.
For a nonzero vector field $X \in \Gamma(D)$, let $X=X_{1}+X_{2}$, where $X_{1}=P_{1}X \in \Gamma(D_{1})$ and $X_{2}=P_{2}X \in \Gamma(D_{2})$.
Because $D_{1}$ is invariant, then $JX_{1}=TX_{1}$ and using the property of the metallic structure (\ref{e1}), we obtain:
$$pTX_{1}+qX_{1}=pJX_{1}+qX_{1}=J^{2}X_{1}=T^{2}X_{1}=\lambda(pTX_{1}+qX_{1})$$ which implies $(pTX_{1}+qX_{1})(\lambda -1)=0$. Because $\lambda \in [0, 1)$, we obtain $TX_{1} = -\frac{q}{p}X_{1}$ and we get $X_{1}=0$ ($\frac{q^{2}}{p^{2}} \neq0$ because $p$ and $q$ are nonzero natural numbers).  Thus, we obtain $X \in \Gamma(D_{2})$ and
$D\subseteq D_{2}$, which implies $D=D_{2}$. Therefore $D_{1}=D^{\bot}$.

Conversely, if there exists a real number $\lambda\in[0,1)$ such that we have $T^2X=\lambda(pTX+qX)$, for any $X \in \Gamma(D)$, it follows $\cos^2(\theta(X))=\lambda$ which implies that $\theta(X)=\arccos(\sqrt{\lambda})$ does not depend on $X$.

We can consider the orthogonal direct sum $TM=D \oplus D^{\bot}$. For $Y \in \Gamma(D^{\bot}):=\Gamma(D_{1})$ and $X \in \Gamma(D)$ (with $D:=D_{2}$), we have
$$\overline{g}(X,J^{2}Y)=\overline{g}(X,T(JY))=\overline{g}(TX,JY)=\overline{g}(TX,TY)=$$ $$=\overline{g}(T^{2}X,Y)=\lambda[p \overline{g}(TX,Y)+q \overline{g}(X,Y)].$$
From $\overline{g}(X,J^{2}Y)=p \overline{g}(X,JY)+q \overline{g}(X,Y)$ and $\overline{g}(X,Y)=0$, we obtain $\overline{g}(X,JY)=\lambda \overline{g}(X,TY)$ and this implies $(1-\lambda)TY \in \Gamma(D^{\bot})$ and $TY \in \Gamma(D^{\bot})$. Thus $JY \in \Gamma(D^{\bot})$, for any $X \in \Gamma(D^{\bot})$ and we obtain that $ D^{\bot}$ is an invariant distribution.
\end{proof}

\begin{remark}
An immersed submanifold $M$ of the Golden Riemannian manifold $(\overline{M},\overline{g},J)$ is a semi-slant submanifold in $\overline{M}$ if and only if there exists a constant $\lambda \in [0, 1)$ such that $$D=\{ X \in \Gamma(TM) | T^{2}X= \lambda(TX+X)\}$$ is a distribution and $NX=0$, for any $X \in \Gamma(TM)$ orthogonal to $D$.
\end{remark}

\textbf{Example 1:}
Let $\mathbb{R}^{7}$ be the Euclidean space endowed with the usual Euclidean metric $<\cdot,\cdot>$.
Let $f: M \rightarrow \mathbb{R}^{7}$ be the immersion given by:
$$f(u,t_{1},t_{2})=(u \cos t_{1}, u \sin t_{1}, u \cos t_{2}, u \sin t_{2}, u, t_{1}, t_{2}),$$
where $M :=\{(u,t_{1},t_{2})/ u>0, t_{1}, t_{2} \in [0, \frac{\pi}{2}]\}$.

We can find a local orthonormal frame on $TM$ given by:
 $$Z_{1}= \cos t_{1} \frac{\partial}{\partial x_{1}} + \sin t_{1} \frac{\partial}{\partial x_{2}} +
 \cos t_{2} \frac{\partial}{\partial x_{3}} +  \sin t_{2} \frac{\partial}{\partial x_{4}} + \frac{\partial}{\partial y_{1}}, $$
  $$Z_{2}= -u\sin t_{1} \frac{\partial}{\partial x_{1}} + u \cos t_{1} \frac{\partial}{\partial x_{2}} + \frac{\partial}{\partial y_{2}}, \quad
  Z_{3}= -u\sin t_{2} \frac{\partial}{\partial x_{3}} + u \cos t_{2} \frac{\partial}{\partial x_{4}} + \frac{\partial}{\partial y_{3}}.$$

We define the metallic structure $J : \mathbb{R}^{7} \rightarrow \mathbb{R}^{7} $ given by:
 \begin{equation}\label{e72}
 J(\frac{\partial}{\partial x_{i}},\frac{\partial}{\partial y_{j}})=
 (\sigma \frac{\partial}{\partial x_{1}},\sigma \frac{\partial}{\partial x_{2}},\overline{\sigma}\frac{\partial}{\partial x_{3}},\overline{\sigma} \frac{\partial}{\partial x_{4}},\overline{\sigma} \frac{\partial}{\partial y_{1}}, \sigma \frac{\partial}{\partial y_{2}},\overline{\sigma}\frac{\partial}{\partial y_{3}}),
  \end{equation}
 for $i \in \{1,2,3,4\}$ and  $j \in \{1,2,3\}$ where $\sigma:=\sigma_{p,q}=\frac{p+\sqrt{p^{2}+4q}}{2}$ is the metallic number ($p, q \in N^{*}$) and $\overline{\sigma}=p-\sigma$.  Thus, we obtain:
 $$JZ_{1}= \sigma\cos t_{1} \frac{\partial}{\partial x_{1}} + \sigma \sin t_{1} \frac{\partial}{\partial x_{2}} +
 \overline{\sigma}\cos t_{2} \frac{\partial}{\partial x_{3}} +  \overline{\sigma}\sin t_{2} \frac{\partial}{\partial x_{4}} +\overline{\sigma} \frac{\partial}{\partial y_{1}}, $$
  $$ JZ_{2}= \sigma Z_{2}, \quad JZ_{3}= \overline{\sigma} Z_{3}.$$
 We can verify that
 $\|Z_{1}\|^{2}=3, \quad \|Z_{2}\|^{2}=\|Z_{3}\|^{2}=u^{2}+1$ and
  $$ \|JZ_{1}\|^{2}= \sigma^{2}+2\overline{\sigma}^{2},\quad \|JZ_{2}\|^{2}=\sigma^{2}(u^{2}+1), \quad \|JZ_{3}\|^{2}=\overline{\sigma}^{2}(u^{2}+1).$$

 On the other hand, we have $<JZ_{1}, Z_{1}>=\sigma + 2\overline{\sigma}$ and $<JZ_{i},Z_{j}>=0,$ for any $i \neq j$, where $i,j \in \{1,2,3\}$.
  We remark that:
 $$ \cos \theta = \frac{<JZ_{1},Z_{1}>}{\|Z_{1}\|\cdot \|JZ_{1}\|}=\frac{\sigma+ 2\overline{\sigma}}{\sqrt{3(\sigma^{2}+2 \overline{\sigma}^{2})}}.$$
 Moreover $\frac{\sigma+ 2\overline{\sigma}}{\sqrt{3(\sigma^{2}+2\overline{\sigma}^{2})}}<1$ and it is a constant.
We define the distributions $D_{1}=span\{Z_{2}, Z_{3}\}$ and $D_{2}=span\{Z_{1}\}$. We have $J(D_{1})\subset D_{1}$ (i.e. $D_{1}$ is an invariant distribution with respect to $J$).
The Riemannian metric tensor of $D_{1} \oplus D_{2}$ is given by
$
g=3du^{2} + (u^{2}+1)(d t_{1}^{2}+d t_{2}^{2}).
$
Thus, $D_{1} \oplus D_{2}$ is a warped product semi-slant submanifold in the metallic Riemannian manifold  $(\mathbb{R}^{7}, <\cdot,\cdot>, J)$
with the slant angle $\arccos(\frac{\sigma+ 2\overline{\sigma}}{\sqrt{3(\sigma^{2}+2\overline{\sigma}^{2})}})$.

If $J$ is the Golden structure $J : \mathbb{R}^{7} \rightarrow \mathbb{R}^{7}$ given by:
 \begin{equation}\label{e73}
 J(\frac{\partial}{\partial x_{i}},\frac{\partial}{\partial y_{j}})=
 (\phi \frac{\partial}{\partial x_{1}},\phi \frac{\partial}{\partial x_{2}},\overline{\phi}\frac{\partial}{\partial x_{3}},\overline{\phi} \frac{\partial}{\partial x_{4}},\overline{\phi} \frac{\partial}{\partial y_{1}}, \phi \frac{\partial}{\partial y_{2}},\overline{\phi}\frac{\partial}{\partial y_{3}}),
  \end{equation}
for $i \in \{1,2,3,4\}$ and  $j \in \{1,2,3\}$, where $\phi:=\frac{1+\sqrt{5}}{2}$ is the Golden number and $\overline{\phi}=1-\phi$, in the same manner we obtain:
 $$ \cos \theta = \frac{<JZ_{1},Z_{1}>}{\|Z_{1}\|\cdot\|JZ_{1}\|}=\frac{\phi+2\overline{\phi}}{\sqrt{3(\phi^{2}+2\overline{\phi}^{2})}}.$$  We define the distributions $D_{1}=span\{Z_{2}, Z_{3}\}$ and $D_{2}=span\{Z_{1}\}$. We obtain that $D_{1} \oplus D_{2}$ is a warped product semi-slant submanifold in the Golden Riemannian manifold  $(\mathbb{R}^{7}, <\cdot,\cdot>, J)$, with the slant angle $\arccos(\frac{\phi+2\overline{\phi}}{\sqrt{3(\phi^{2}+2\overline{\phi}^{2})}})$.

\textbf{Example 2:}
Let $M:=\{(u,\alpha_{1},\alpha_{2},...,\alpha_{n})/ u > 0, \alpha_{i} \in [0, \frac{\pi}{2}], i \in \{1,...,n\} \}$ and $f:M\rightarrow \mathbb{R}^{3n+1}$ is the immersion given by:
\begin{equation} \label{e74}
f(u,\alpha_{1},...,\alpha_{n})=(u \cos\alpha_{1},...,u \cos\alpha_{n},u \sin\alpha_{1},...,u \sin\alpha_{n},\alpha_{1},...,\alpha_{n},u).
\end{equation}

We can find a local orthonormal frame of the submanifold $TM$ in $\mathbb{R}^{3n+1}$, spanned by the vectors:
\begin{equation} \label{e75}
Z_{0}=\sum_{j=1}^{n} (\cos \alpha_{j}\frac{\partial}{\partial x_{j}}+\sin \alpha_{j}\frac{\partial}{\partial x_{n+j}})+\frac{\partial}{\partial x_{3n+1}},
\end{equation}
and
\begin{equation} \label{e750}
Z_{i}= -u \sin \alpha_{i}\frac{\partial}{\partial x_{i}}+u \cos \alpha_{i}\frac{\partial}{\partial x_{n+i}}+\frac{\partial}{\partial x_{2n+i}},
\end{equation}
for any $i \in \{1,...,n\}$.

We remark that $\|Z_{0}\|^{2}=n+1$, $\|Z_{i}\|^{2}=u^{2}+1$, for any $i \in \{1,...,n\}$,
 $Z_{0} \bot Z_{i}$, for any $i \in \{1,...,n\}$ and $Z_{i} \bot Z_{j}$, for $i \neq j$, where $i, j \in \{1,...,n\}$.

Let $J:\mathbb{R}^{3n+1}\rightarrow \mathbb{R}^{3n+1}$ be the $(1,1)$-tensor field defined by:
\begin{equation} \label{e76}
J(X^{1},...,X^{3n},X^{3n+1})=(\sigma X^{1},...,\sigma X^{3n},\overline{\sigma}X^{3n+1}),
\end{equation}
where $\sigma:=\sigma_{p,q}$ is the metallic number and $\overline{\sigma}=p-\sigma$. It is easy to verify that $J$ is a metallic structure on $\mathbb{R}^{3n+1}$ (i.e. $J^{2}=pJ+qI$).

Moreover, the metric $\overline{g}$, given by the scalar product $<\cdot,\cdot>$ on $\mathbb{R}^{3n+1}$, is $J$ compatible and $(\mathbb{R}^{3n+1},\overline{g},J)$ is a metallic Riemannian manifold.

From (\ref{e75}) and (\ref{e750}) we get:
$$
JZ_{0}=\sigma \sum_{j=1}^{n}(\cos \alpha_{j}\frac{\partial}{\partial x_{j}}+\sin \alpha_{j}\frac{\partial}{\partial x_{n+j}})+ \overline{\sigma}\frac{\partial}{\partial x_{3n+1}}
$$
and
$$JZ_{i}= \sigma ( -u \sin \alpha_{i}\frac{\partial}{\partial x_{i}} + u  \cos \alpha_{i}\frac{\partial}{\partial x_{n+i}}+ \frac{\partial}{\partial x_{2n+i}})=\sigma Z_{i},$$
for any $i \in \{1,...,n\}$.

We can verify that $JZ_{0}$ is orthogonal to $span\{Z_{1},...,Z_{n}\}$ and
\begin{equation} \label{e77}
\cos(\widehat{JZ_{0},Z_{0}}) = \frac{n\sigma +\overline{\sigma}}{\sqrt{(n+1)(n\sigma^{2} +\bar{\sigma}^{2})}}.
\end{equation}
If we consider the distributions $D_{1}=span \{Z_{i} / i \in \{1,...,n\}\}$ and $D_{2}=span\{Z_{0}\}$, then $D_{1} \oplus D_{2}$ is a semi-slant submanifold in the metallic Riemannian manifold $(\mathbb{R}^{3n+1},<\cdot,\cdot>,J)$, with the Riemannian metric tensor $$g=(n+1) du^{2} + (u^{2}+1)\sum_{j=1}^{n}d\alpha_{j}^{2}.$$

%%%%%%%%%%%%%%%%%%%%%%%

\section{On the integrability of the distributions of semi-slant submanifolds }

In this section we investigate the conditions for the integrability of the distributions of semi-slant submanifolds in metallic (or Golden) Riemannian manifolds.

\begin{theorem}
If $M$ is a semi-slant submanifold of a locally metallic (or locally Golden) Riemannian manifold $(\overline{M},\overline{g},J)$, then:\\
(i) the distribution $D_{1}$ is integrable if and only if
\begin{equation}\label{e78}
(\nabla_{Y}u_{\alpha})(X)=(\nabla_{X}u_{\alpha})(Y),
\end{equation}
 for any $X,Y \in \Gamma(D_{1})$;\\
(ii) the distribution $D_{2}$ is integrable if and only if
\begin{equation}\label{e79}
P_{1}(\nabla_{X}TY-\nabla_{Y}TX)=\sum_{i=1}^{r}[u_{\alpha}(Y)P_{1}(A_{\alpha}X)-u_{\alpha}(X)P_{1}(A_{\alpha}Y)],
\end{equation}
 for any $X,Y \in \Gamma(D_{2})$.
\end{theorem}

\begin{proof}
(i) For $X, Y \in \Gamma(D_{1})$, we have $X=P_{1}X$ and $Y=P_{1}Y$. The distribution $D_{1}$ is integrable if and only if $[X,Y] \in \Gamma(D_{1})$, which is equivalent to $N([X,Y])=0$, for any $X, Y \in \Gamma(D_{1})$. From $J(D_{1})\subseteq D_{1}$ we obtain $NX=NY=0$ and from (\ref{e25})(i) we get $u_{\alpha}(X)l_{\alpha\beta}(Y)=u_{\alpha}(Y)l_{\alpha\beta}(X)=0$. Thus, using (\ref{e44}) we have the distribution $D_{1}$ is integrable if and only if  (\ref{e78}) holds.

(ii) For $X, Y \in \Gamma(D_{2})$, we have $X=P_{2}X$, $Y=P_{2}Y$. The distribution $D_{2}$ is integrable if and only if $[X,Y] \in \Gamma(D_{2})$, which is equivalent to $P_{1}T([X,Y])=0$. Thus, from (\ref{e43}), we obtain the distribution $D_{2}$ is integrable if and only if (\ref{e79}) holds.
 \end{proof}

\begin{remark}
 If $M$ is a semi-slant submanifold of a locally metallic (or locally Golden) Riemannian manifold $(\overline{M},\overline{g},J)$, then:\\
(i) the distribution $D_{1}$ is integrable if and only if
\begin{equation}\label{e80}
h(X,TY)=h(TX,Y),
\end{equation}
 for any $X,Y \in \Gamma(D_{1})$;\\
(ii) the distribution  $D_{1}$ is integrable if and only if the shape operator of $M$ satisfies:
 \begin{equation}\label{e81}
 J A_{V}X = A_{V}JX,
 \end{equation}
 for any  $X \in \Gamma(D_{1})$ and $V \in \Gamma(T^{\bot}M)$;  \\
(iii) the distribution $D_{2}$ is integrable if and only if
 \begin{equation}\label{e82}
 P_{1}(\nabla_{X}TY-\nabla_{Y}TX)=P_{1}(A_{NY}X-A_{NX}Y),
 \end{equation}
 for any $X,Y \in \Gamma(D_{2})$.
\end{remark}

\begin{proof}
(i) For any $X,Y \in \Gamma(D_{1})$, we have $[X,Y] \in \Gamma(D_{1})$ if and only if $N([X,Y])=0$ and from (\ref{e37}), we obtain (\ref{e80}).

(ii) For any $X,Y \in \Gamma(D_{1})$ and any $V \in \Gamma(T^{\bot}M)$, from (\ref{e31}) and (\ref{e2}) we have:
$$g( J A_{V}X - A_{V}JX, Y) = g(h(X,JY)-h(JX,Y),V).$$
From (\ref{e37}) and $NX=NY=0$ (because $J(D_{1})\subseteq D_{1}$) we have $$g( J A_{V}X - A_{V}JX, Y)= g(N([X,Y]),V)=0,$$
for any $X,Y \in \Gamma(D_{1})$ and any $V \in \Gamma(T^{\bot}M)$. Thus, we have $[X,Y] \in \Gamma(D_{1})$.

(iii) For any $X, Y \in \Gamma(D_{2})$, we have $X=P_{2}X$ and $Y=P_{2}Y$. The distribution $D_{2}$ is integrable if and only if $[X,Y] \in \Gamma(D_{2})$, which is equivalent to $T([X,Y]) \in \Gamma(D_{2})$  or  $P_{1}T([X,Y])=0$. Thus, from (\ref{e36}), we obtain that $D_{2}$ is integrable if and only if (\ref{e82}) holds.
\end{proof}

\begin{theorem}
Let $M$  be a semi-slant submanifold of a locally metallic (or locally Golden) Riemannian manifold ($\overline{M},\overline{g},J)$. If $\nabla T =0$, then the distributions $D_{1}$ and $D_{2}$ are integrable.
\end{theorem}

\begin{proof}
First of all, we consider $X , Y \in \Gamma(D_{1})$ and we proof $[X,Y] \in \Gamma(D_{1})$.
For any $Y \in \Gamma(D_{1})$, we get $NY =0$ and using $\nabla T =0$ in (\ref{e40})(i) we obtain  $th(X,Y)=0$, for any $X, Y \in \Gamma(D_{1})$, which implies $Jh(X,Y) = n h(X,Y)$. From
$$\overline{g}(th(X,Y),Z)=\overline{g}(Jh(X,Y),Z)=\overline{g}(h(X,Y),JZ)$$
 and (\ref{e63}) we get $\overline{g}(h(X,Y),NP_{2}Z)=0$, for any $X, Y \in \Gamma(D_{1})$ and $Z \in \Gamma(TM)$. Thus, from (\ref{e1}) and (\ref{e2}) we get
$$\overline{g}(Jh(X,Y),JZ)=\overline{g}(J^{2}h(X,Y),Z)=p \overline{g}(J h(X,Y),Z)+q \overline{g}(h(X,Y),Z)=0,$$
for any $X, Y \in \Gamma(D_{1})$ and $Z \in \Gamma(TM)$.
Moreover, for $Z=\nabla_{X}Y$, we obtain:
$$ 0=\overline{g}(Jh(X,Y),NP_{2}\nabla_{X}Y)=\overline{g}(\overline{\nabla}_{X}JY,NP_{2}\nabla_{X}Y) -\overline{g}(J\nabla_{X}Y ,NP_{2}\nabla_{X}Y),$$
which implies:
\begin{equation}\label{e83}
\overline{g}(h(X,JY),NP_{2}\nabla_{X}Y) =\overline{g}(NP_{2}\nabla_{X}Y ,NP_{2}\nabla_{X}Y).
\end{equation}
On the other hand, from (\ref{e67}) and  (\ref{e83}) we have:
\begin{equation}\label{e84}
\overline{g}(h(X,JY),NP_{2}\nabla_{X}Y)=\sin^{2}\theta [p \overline{g}(TP_{2}\nabla_{X}Y ,P_{2}\nabla_{X}Y)+ q \overline{g}(P_{2}\nabla_{X}Y ,P_{2}\nabla_{X}Y)].
\end{equation}
Using (\ref{e83}), (\ref{e16}) and $JY=TY$, for any $Y \in \Gamma(D_{1})$,  we obtain:
 \begin{equation}\label{e840}
\overline{g}(h(X,JY),NP_{2}\nabla_{X}Y)= \overline{g}(t h(X,TY),P_{2}\nabla_{X}Y)=0.
\end{equation}
Thus, from (\ref{e83}) and (\ref{e840}) we have:
$$\sin^{2}\theta [p \overline{g}(TP_{2}\nabla_{X}Y ,P_{2}\nabla_{X}Y)+ q \overline{g}(P_{2}\nabla_{X}Y ,P_{2}\nabla_{X}Y)]=0.$$
From $\theta \neq 0$,  we obtain $p \overline{g}(TP_{2}\nabla_{X}Y ,P_{2}\nabla_{X}Y)+ q \overline{g}(P_{2}\nabla_{X}Y ,P_{2}\nabla_{X}Y)=0$
and using $\overline{g}(TP_{2}\nabla_{X}Y ,P_{2}\nabla_{X}Y)=\overline{g}(JP_{2}\nabla_{X}Y ,P_{2}\nabla_{X}Y)$ we have:
$$
\overline{g}(J^{2}P_{2}\nabla_{X}Y ,P_{2}\nabla_{X}Y)=p\overline{g}(JP_{2}\nabla_{X}Y ,P_{2}\nabla_{X}Y)+q\overline{g}(P_{2}\nabla_{X}Y ,P_{2}\nabla_{X}Y)=0,
$$
which implies $\overline{g}( J(P_{2}\nabla_{X}Y),J(P_{2}\nabla_{X}Y))=0$. Thus, we get $J(P_{2}\nabla_{X}Y)=0$ and we obtain $P_{2}\nabla_{X}Y=0$. In conclusion, $\nabla_{X}Y \in \Gamma(D_{1})$ for any $X, Y \in \Gamma(D_{1})$ and this implies $[X,Y] \in \Gamma(D_{1})$. Thus, the distribution $D_{1}$ is integrable.

Moreover, because $D_{2}$ is orthogonal to $D_{1}$ and ($\overline{M},\overline{g})$ is a Riemannian manifold, we obtain the integrability of the distribution $D_{2}$.
\end{proof}

In the next propositions, we consider semi-slant submanifolds in the locally metallic (or locally Golden) Riemannian manifolds and we find some conditions for these submanifolds to be $D_{1} - D_{2}$ mixed totally geodesic (i.e. $h(X,Y)=0$, for any $X \in \Gamma(D_{1})$ and $Y \in \Gamma(D_{2})$), in a similar manner as in the case of semi-slant submanifolds in locally product manifolds (\cite{Li&Liu}).

\begin{prop}
If $M$ is a semi-slant submanifold of a locally metallic (or locally Golden) Riemannian manifold $(\overline{M},\overline{g},J)$, then $M$ is a $D_{1}-D_{2}$ mixed totally geodesic submanifold if and only if $A_{V}X \in \Gamma(D_{1})$ and  $A_{V}Y \in \Gamma(D_{2})$, for any $X \in \Gamma(D_{1})$, $Y \in \Gamma(D_{2})$ and $V \in \Gamma(T^{\bot}M)$.
\end{prop}

\begin{proof}
From (\ref{e31}) we remark that $M$ is a $D_{1}-D_{2}$ mixed totally geodesic  submanifolds in the locally metallic (or locally Golden) Riemannian manifolds if and only if $g(A_{V}X,Y)=g(A_{V}Y,X)=0$, for any $X \in \Gamma(D_{1}), Y \in \Gamma(D_{2})$ and $V \in \Gamma(T^{\bot}M)$, which is equivalent to $A_{V}X \in \Gamma(D_{1})$ and  $A_{V}Y \in \Gamma(D_{2})$.
\end{proof}

\begin{prop}\label{p50}
Let $M$  be a proper semi-slant submanifold of a locally metallic (or locally Golden) Riemannian manifold ($\overline{M},\overline{g},J)$.
If $M$ is a $D_{1} - D_{2}$ mixed totally geodesic submanifold then $(\overline{\nabla}_{X}N)Y=0$, for any $X \in \Gamma(D_{1})$ and $Y \in \Gamma(D_{2})$.
\end{prop}

\begin{proof}
 If $M$ is a $D_{1}-D_{2}$ mixed geodesic submanifold, then $nh(X,Y)=h(X,TY)$ and using (\ref{e40})(ii), we obtain $(\overline{\nabla}_{X}N)Y=0$, for any $X \in \Gamma(D_{1})$, $Y \in \Gamma(D_{2})$.
\end{proof}

\begin{theorem}\label{t5}
Let $M$ be a proper semi-slant submanifold of a locally metallic (or locally Golden) Riemannian manifold ($\overline{M},\overline{g},J)$.
If $(\overline{\nabla}_{X}N)Y=0$, for any $X \in \Gamma(D_{1})$, $Y \in \Gamma(D_{2})$ and $h(X,Y)$ is not an eigenvector of the tensor field $n$ with the eigenvalue $-\frac{q}{p}$, then $M$ is a $D_{1} - D_{2}$ mixed totally geodesic submanifold.
\end{theorem}

\begin{proof}
If $M$ is a semi-slant submanifold of a locally metallic (or locally Golden) Riemannian manifold ($\overline{M},\overline{g},J)$
and $(\overline{\nabla}_{X}N)Y=0$ then, from (\ref{e40})(ii), we get:
\begin{equation}\label{e85}
n^{2}h(X,Y)=nh(X,TY)=h(X,T^{2}Y),
\end{equation}
for any $X \in \Gamma(D_{1})$ and $Y \in \Gamma(D_{2})$, where $nV:=(J V)^{\perp}$ for any $V \in \Gamma(T^{\perp}M)$. By using $TY=TP_{2}Y$, for any $Y \in \Gamma(D_{2})$ and (\ref{e68}), we obtain:
 \begin{equation}\label{e86}
n^{2}h(X,Y)=\cos^{2}\theta (p nh(X,Y)+q h(X,Y)),
\end{equation}
where $\theta$ is the slant angle of the distribution $D_{2}$.
Using $TX=TP_{1}X=JX$, for any $X \in \Gamma(D_{1})$, we obtain:
$$n^{2}h(X,Y)=h(T^{2}X,Y)=h(J^{2}X,Y)=h(p JX+q X,Y))= p h(T X,Y)+q h(X,Y).$$ Thus, we get:
\begin{equation}\label{e87}
n^{2}h(X,Y)=p nh(X,Y)+q h(X,Y),
\end{equation}
for any $X \in \Gamma(D_{1})$ and $Y \in \Gamma(D_{2})$. From (\ref{e86}), (\ref{e87}) and $\cos^{2}\theta \neq 0$ ($D_{2}$ is a proper semi-slant distribution), we remark that $ nh(X,Y) = -\frac{q}{p}h(X,Y)$ and this implies  $h(X,Y)=0$, for any $X \in \Gamma(D_{1})$ and $Y \in \Gamma(D_{2})$, because $h(X,Y)$ is not an eigenvector of $n$ with the eigenvalue $-\frac{q}{p}$. Thus $M$ is a $D_{1}-D_{2}$ mixed totally geodesic submanifold in the locally metallic (or locally Golden) Riemannian manifold ($\overline{M},\overline{g},J)$.
\end{proof}

In a similar manner as in (\cite{Atceken1}, Theorem 4.8), we get:

\begin{prop}
Let $M$ be a semi-slant submanifold in a locally metallic (or locally Golden) Riemannian manifold ($\overline{M},\overline{g},J$). Then $N$ is parallel if and only if the shape operator $A$ verifies:
\begin{equation}\label{e88}
 A_{nV}X=TA_{V}X=A_{V}TX,
\end{equation}
 for any  $X \in \Gamma(TM)$ and $V \in \Gamma(T^{\bot}M)$.
\end{prop}

\begin{proof}
From (\ref{e2}), we get $\overline{g}(nh(X,Y),V)=\overline{g}(Jh(X,Y),V)=\overline{g}(h(X,Y),nV)$, for any $X, Y \in \Gamma(TM)$, $V \in \Gamma(T^{\bot}M)$. Thus, by using (\ref{e40})(ii), we obtain
$$\overline{g}((\overline{\nabla}_{X}N)Y,V)=\overline{g}(h(X,Y),nV)-\overline{g}(h(X,TY),V)=\overline{g}(A_{nV}X,Y)-\overline{g}(A_{V}X,TY),$$
for any $X, Y \in \Gamma(TM)$, $V \in \Gamma(T^{\bot}M)$ and we have
\begin{equation}\label{e89}
\overline{g}((\overline{\nabla}_{X}N)Y,V)=\overline{g}(A_{nV}X-TA_{V}X,Y)=\overline{g}(A_{nV}Y-A_{V}TY,X),
\end{equation}
for any $X, Y \in \Gamma(TM)$, $V \in \Gamma(T^{\bot}M)$. Thus, from  (\ref{e89}) we obtain (\ref{e88}).
\end{proof}

\begin{theorem}
Let $M$ be a proper semi-slant submanifold in a locally metallic (or locally Golden) Riemannian manifold ($\overline{M},\overline{g},J$). If the shape operator $A$ verifies
$A_{nV}X=TA_{V}X=A_{V}TX$, for any  $X \in \Gamma(TM)$, $V \in \Gamma(T^{\bot}M)$ and $h(X,Y)$ is not an eigenvector of the tensor field $n$ with the eigenvalue $-\frac{q}{p}$ then, $M$ is a $D_{1} - D_{2}$ mixed totally geodesic submanifold.
\end{theorem}

\begin{proof}
If $A_{nV}X=TA_{V}X=A_{V}TX$, for any  $X \in \Gamma(TM)$, $V \in \Gamma(T^{\bot}M)$ then, from (\ref{e89}) we obtain $(\overline{\nabla}_{X}N)Y=0$ for any $X, Y \in \Gamma(TM)$ and using the Theorem 5 we obtain that $M$ is a $D_{1}-D_{2}$ mixed totally geodesic submanifold.
\end{proof}

%\textbf{Acknowledgements}: The authors want to express their gratitude to the referee of this paper.

\linespread{1}
Cristina E. Hretcanu, \\ Stefan cel Mare University of Suceava, Romania, e-mail: criselenab@yahoo.com, cristina.hretcanu@fia.usv.ro\\
Adara M. Blaga, \\ West University of Timisoara, Romania, e-mail: adarablaga@yahoo.com.

\end{document}